\documentclass[reqno, draft]{amsart}
\usepackage{amsfonts}

\usepackage{amssymb, amscd, amsfonts, amscd,  amsthm}
\usepackage[mathscr]{eucal}

\newtheorem{theorem}{Theorem}[section]
\newtheorem{lemma}[theorem]{Lemma}
\newtheorem{corollary}[theorem]{Corollary}
\newtheorem{proposition}[theorem]{Proposition}
\newtheorem{example}[theorem]{Example}

\theoremstyle{definition}
\newtheorem{definition}[theorem]{Definition}
\newtheorem{remark}[theorem]{Remark}

\newcommand{\remove}[1]{}

\DeclareMathOperator{\lcm}{\rm lcm}  
\DeclareMathOperator{\comp}{\rm comp} 

\DeclareMathOperator{\pos}{\rm pos}
\DeclareMathOperator{\Rn}{{\mathbb R^{\it n}}}
\DeclareMathOperator{\R2}{{\mathbb R^{2}}}
\DeclareMathOperator{\den}{\rm den}
\DeclareMathOperator{\conv}{\rm conv}
\DeclareMathOperator{\rank}{\rm rank}
\DeclareMathOperator{\orb}{\rm orb}

 \title[Classifying $GL(2,\mathbb Z) \ltimes \mathbb Z^{2}$-orbits]
{Classifying $GL(2,\mathbb Z) \ltimes \mathbb Z^{2}$-orbits 
by subgroups of $\mathbb R$}

\author{Daniele Mundici}

\date{\today}

\address[D. Mundici]{Department of Mathematics ``Ulisse Dini'' \\
University of Florence\\
Viale Morgagni 67/A \\
I-50134 Florence \\
Italy}

\email{mundici@math.unifi.it}

\begin{document}

\thanks{2000 {\it Mathematics Subject Classification.}
Primary:   37C85.  Secondary: 11B57, 22F05, 37A45, 06F20, 46L80.}
\keywords{Affine group over the integers,  orbit,
$GL(2,\mathbb Z) \ltimes \mathbb Z^{2}$-orbit, 
complete invariant,  rational simplex, regular cone,  
Farey sequence, 
totally ordered abelian group,
AF C$^*$-algebra}

\begin{abstract}
Let  $\mathcal G_2$ denote the affine group
$GL(2,\mathbb Z) \ltimes \mathbb Z^{2}$.
For every  point $x=(x_1,x_2) \in \R2$ let  
$\orb(x)=\{y\in\R2\mid y=\gamma(x)$ for some
$\gamma  \in \mathcal{G}_2 \}$. 
Let $G_{x}$ be the subgroup of the additive group
$\mathbb R$ generated by $x_1,x_2, 1$. 
If    $\rank(G_x)\in \{1,3\}$ then 
$\orb(x)=\{y\in\R2\mid G_y=G_x\}$.
If $\rank(G_x)=2$, knowledge of 
$G_x$ is not   sufficient in general 
to uniquely recover  $\orb(x)$: rather,  $G_x$
classifies precisely  $\max(1,\phi(d)/2)$
different orbits, where $d$ is the denominator
of the smallest positive nonzero rational in $G_x$
and $\phi$ is Euler function.
To get a complete classification, polyhedral geometry 
provides  an integer  $c_x\geq 1$  such that
$\orb(y)=\orb(x) $    iff
$(G_{x},c_{x})=(G_{y},c_{{y}})$.   
\end{abstract}

\maketitle

\section{Introduction} 
Throughout we   let
$\mathcal G_n=GL(n,\mathbb Z) \ltimes \mathbb Z^{n}$ 
denote the group of affine  transformations of the form
$
x\mapsto Ux + t \,\, \mbox{for} \,\, x\in \mathbb R^n,
$
 where $t\in \mathbb Z^n$
 and  $U$ is an integer  $n \times n$ matrix with
 $\det(U)=\pm 1$.   
By an ``orbit''   we will mean a
$\mathcal G_n$-orbit. 
The dimension  of the ambient space  $\Rn$ will always be clear from
the context.
We let $\orb(x)$  denote the
orbit of $x\in \Rn.$  Thus
$$
\orb(x)=\{y\in \Rn\mid y=\gamma(x) \mbox{ for some }\,
\gamma\in \mathcal G_n\}.
$$

This paper is devoted to the classification of
$\mathcal G_2$-orbits:  to 
every point in  $\R2$ we will assign an invariant,
in such a way that two points have the same orbit iff
they have the same invariant.
For each $x=(x_1,\ldots,x_n)\in \Rn$,
the  main invariant considered in this paper is the group
$$
G_x=\mathbb Z+\mathbb Zx_1+\mathbb Zx_2+\dots+\mathbb Zx_n
$$
generated by $1,x_1,x_2,\ldots,x_n$ in the
 additive group
$\mathbb R$.
In Propositions   \ref{proposition:generic} and
  \ref{proposition:rank1}  (also see Remark \ref{remark:generalize})
   it is shown that $G_x$  
 completely classifies  the orbit of any
 $x\in\Rn,\,\,\,n\geq 2,$ whenever $\rank(G_x)$ equals
  1 or   $n+1$.  

 If   $x\in\R2$ and  $\rank(G_x)=2$,
 the invariant  $G_x$ is not complete for $\orb(x).$
 Rather,   $G_x$  classifies precisely  $\max(1,\phi(d)/2)$
different orbits, where $d$ is the denominator
of the smallest positive nonzero rational in $G_x$
and $\phi$ is Euler function. To get a complete
  invariant  we
  assign to $x$ an integer $c_x \geq 1$
and prove in Theorem \ref{theorem:rank2} that
  $\orb(x)=\orb(y)$  iff
$(G_x,c_x)=(G_y,c_y).$   In Corollary \ref{corollary:tutti}
we describe the totality   of pairs  $(G_x,c_x)$ arising
as complete classifiers of the totality  
of orbits  of points of $\R2$. The geometric significance of
$c_x$ is discussed in Section 5.

While the study of the orbits of $\mathcal G_2$  and  of 
its  subgroups   $GL(2,\mathbb Z)$,   $SL(2,\mathbb Z),$
  $SL(2,\mathbb Z)_+$ 
requires techniques from various mathematical areas 
\cite{dan, gui, launog, nog2002, nog2010, wit},
a first possible reason of interest in our 
classification  stems
from  the novel  role played by
  polyhedral geometry  \cite{ewa},
 through the fundamental notion
of  regular cone and regular simplex.
(As is well known, a regular fan is a complex
of regular cones.)  Secondly and thirdly,  in a final section we will 
apply  our classification
to the analysis of germinal ideals of
 lattice-ordered abelian groups with strong order unit
\cite{bigkeiwol},\cite{mun-dcds}
and of the Farey-Stern-Brocot AF C$^*$-algebra,
\cite{mun-adv,boc, eck,mun-mjm}.

\section{Rank 3 and Rank 1}
Fix  $n=1,2,\ldots$.
  In case  $x\in \Rn$  and $\rank(G_x)=n+1$, 
the following  elementary
result
  shows that the
group $G_x$  is a complete classifier of the orbit of $x$.
Its  proof is left to the reader.

 \begin{proposition}
 \label{proposition:generic}
Let  us say that a
 point $x =(x_1,\ldots,x_n) \in \mathbb R^n$  is   
{\rm totally irrational}  if  the $(n+1)$-tuple  $(x_1,\ldots,x_n,1)$
 is rationally independent.  In other words,
 $\rank(G_x)=n+1$. Whenever  $x$
is totally irrational,
$\orb(x)$
 is  completely classified by $G_x$.
Thus in particular, if  $x \in \R2$ and
$\rank(G_x)=3$   then 
$\orb(x)=\{y\in\R2\mid G_y=G_x\}$.
 \end{proposition}

To deal with the remaining cases   
we  introduce some basic tools from
polyhedral geometry \cite{ewa},
\cite{mun-dcds}, featuring the notion of regular cone
and simplex.

 \subsection*{Rational points, simplexes and cones}
A point $y=(y_1,\ldots,y_n) \in \mathbb \Rn$
is said to be {\it rational}  if $y_i\in\mathbb Q$
for each  $i=1,\ldots,n$.  Equivalently,  $\rank(G_x)=1.$
By  the  {\it denominator}  $\den(y)$ of the rational point
  $y=(y_1,\ldots,y_n) $
 we understand the least common denominator 
  of its coordinates. 
Observe that, for each rational point $y\in \mathbb R^n$,
  $G_y$ is the
 group of all integer multiples of 
 $1/\den(y)$.
 A simplex
 $T\subseteq \Rn$ is said to be {\it rational}
if   each vertex  
of $T$  is a  rational point.  
A {\it rational hyperplane} $H \subseteq \mathbb R^n$
is  a set of the form 
$ 
H =\{z \in \mathbb R^{n}   \mid  h\circ z = k\},
$ 
    for some nonzero vector
     $h\in \mathbb Q^n$ and $k\in \mathbb Q$.
 %
 %
 %
   A {\it rational affine space} $A$ in  $\mathbb R^n$
is  the intersection
of rational hyperplanes in $\mathbb R^n$.
 If the origin belongs to $A$
we say that $A$ is {\it homogeneous}.

%
%


 Following \cite[p.6]{ewa},
given  vectors $v_1,\ldots,v_s \in \mathbb R^n$ 
 we write
$$
\pos[v_1 ,\ldots, v_s]=
\mathbb R_{\geq 0}\,v_1
+\cdots+ \mathbb R_{\geq 0}\, v_s
$$
for their {\it positive hull}
in $\mathbb R^n$.
A nonzero integer vector $v\in \mathbb Z^n$ is said to be
{\it primitive}  (``visible'' in \cite{harwri})  if it is 
minimal,  as a nonzero   integer vector,  along its {\it ray} 
$\pos[v]$. 
Further, \cite[p.146]{ewa},  for 
  $t=1,2,\ldots,n$,
a  {\it $t$-dimensional
rational simplicial cone} in $\mathbb R^n$ is a set $C
\subseteq \mathbb R^n$ of the form 
$
    C=  \pos[w_1 ,\ldots, w_t],
$
for linearly independent primitive integer vectors 
$w_1 ,\ldots, w_t \in \mathbb Z^n$.  
The vectors $w_1 ,\ldots, w_t$ are called the
{\it primitive generating vectors} of $C$.  They are uniquely
determined by $C$.  
By a {\it face } of $C$ we mean the positive hull
of a subset  of $\{w_1 ,\ldots, w_t\}.$
For the sake of completeness we stipulate that
the face of  $C$  determined by the empty
set is the singleton  $\{0\}$.  This is the only
zerodimensional cone in  $\mathbb R^n$. 
By definition, the  {\it homogeneous correspondent}
of $y$ is the primitive integer  vector 
 $$
 \tilde y = (\den(y)\cdot y_1,
 \ldots, \,\den(y)\cdot y_n,\,\,\den(y)) =
  \den(y)(y,1)\in \mathbb Z^{n+1}.
  $$
Given  $k=0,1,\ldots$ and
          a rational $k$-simplex 
	$T = \conv(v_0, \ldots, v_k) \subseteq
\mathbb R^{n}$,  
we denote by
$T^\uparrow$
the positive hull in $\mathbb R^{n+1}$
of the homogeneous correspondents of the
vertices of $T$.
 In symbols, 
$$
T^{\uparrow}  =\pos[\tilde v_0,
\ldots,\tilde v_k]
= \mathbb R_{\geq 0}\tilde v_0+\dots
+ \mathbb R_{\geq 0}\tilde v_k\,\,\subseteq \mathbb R^{n+1}. 
$$
We say that $T^\uparrow$ is the
{\it  (rational simplicial) cone  of $T$}.
Note that $\dim(T^{\uparrow})=k+1.$

\subsection*{Regularity}
 In rational  polyhedral topology the following
notion is fundamental, \cite[p.146-147]{ewa}:

Let  $m=1,2,\ldots$  and  $t=1,\ldots,m.$
Then a $t$-dimensional  rational simplicial cone
$C  =
\pos[d_1 ,\ldots, d_t] \subseteq \mathbb R^m$
is said to be  {\it regular} 
if the set $\{d_1 ,\ldots, d_t\}$ of its
primitive generating vectors
can be extended to a basis of
the free abelian group $\mathbb Z^m$ of integer
points in $\mathbb R^m$. 
A rational simplex  
$T\subseteq \Rn$
is said to be  {\it regular}  ({\it unimodular} in \cite{mun-dcds})
if so is its associated cone
    $T^\uparrow\subseteq \mathbb R^{n+1}.$

\medskip
The following elementary lemma will be our key tool
to construct maps 
$\gamma\in \mathcal G_n$:

\begin{lemma}
\label{lemma:maps}
Let $T=\conv(v_1,\ldots,v_{n+1})$ and
$T'=\conv(v'_1,\ldots,v'_{n+1})$ be
regular $n$-simplexes in $\Rn$, with
$\den(v_i)=\den(v'_i)$ for all $i=1,\ldots, n+1$.  
Then for a unique  $\gamma\in \mathcal G_n$
we have  $\gamma(v_1)=v'_1,\ldots,\gamma(v_{n+1})=v'_{n+1}.$
 \end{lemma}
 
 \begin{proof}
The vectors  $\tilde v_1,\ldots,\tilde v_{n+1}$
constitute a basis of the free abelian group $\mathbb Z^{n+1}$,
and so do the vectors 
 $\tilde v'_1,\ldots,\tilde v'_{n+1}$. Let $V$ be the
 $(n+1)\times(n+1)$ matrix whose columns are
 $\tilde v_1,\ldots,\tilde v_{n+1}$. Let similarly $V'$
 be the matrix with columns  $\tilde v'_1,\ldots,\tilde v'_{n+1}$.
 There is a unique  (automatically integer and unimodular)
 $(n+1)\times(n+1)$   matrix $M$ such that   $MV=V'$.
 The assumption $\den(v_i)=\den(v'_i)$ for all $i$,   is to the
 effect that the bottom row of $M$ has the form
 $0, \ldots, 0,1$  with $n$ zeros. From $M$ one 
 immediately obtains the desired $\gamma\in \mathcal G_n$.  
\end{proof}

For each $d=1,2,\ldots,$  in the notation of   \cite[3.1]{harwri},
we let 
 $\mathfrak F_d$ be the $d$th {\it Farey sequence}, the ordered   list
of all rational numbers in $[0,1]$  of denominator $\leq d.$

\bigskip

We are now in a position
to  show that  $G_x$  is a complete invariant
 for orbits of  points  $x\in \R2$ with
$\rank(G_x)=1$:

\begin{proposition}
\label{proposition:rank1}
If   $x \in \R2$ and
$\rank(G_x)=1$ then   
$\orb(x)=\{x'\in\R2\mid G_{x'}=G_x\}$.
\end{proposition}

\begin{proof} Trivially,
$\orb(x)\subseteq \{x'\in\R2\mid G_{x'}=G_x\}$.
For the converse inclusion, let
  $d=\den(x).$ Then $G_x=\mathbb Z\frac{1}{d}.$  
If  $x'\in \R2$
satisfies $G_{x'}=G_x$ then  $x'$ is a rational point and $\den(x')=d.$
If $d=1$ a translation by an integer vector will do.
If $d>1$ it is  no loss of generality to assume that $x$ and $x'$
both lie  in the
unit square $[0,1]^2$.
On the  ray    $\mathbb R_{\geq 0}x$, let   
$v$ be  the nonzero integer 
vector nearest to the origin. 
Let $O,e_1,e_2$  respectively denote the origin in $\R2$,
the point $(1,0)$ and the point $(0,1)$. 
Then, as is well known,
some affine transformation   $\alpha\in GL(2,\mathbb Z)$
satisfies  $\alpha(v)=e_1.$  
As a consequence,  $\alpha$  sends $x$ to some rational point
$r=(p/d,0)$ lying on 
the unit interval  $\conv(O,e_1)$.
 Similarly, some
affine transformation
$\beta \in GL(2,\mathbb Z)$  sends $x'$ to some
point $r'=(p'/d,0)\in \conv(O,e_1)$.  Let 
 $q/c$ and $s/t$ be the two neighbors (=contiguous elements) of
 $r$  in  the Farey sequence
 $\mathfrak F_d$,  recalling that   $d>1$. 
  Both denominators  $c$ and $t$
 are strictly smaller than $d$. If $d\geq 3,$ then
 $c\not=t$  and we may assume  $c<t.$
 If $d=2$  we have  $c=t=1.$
Let the integer vector $w\in \R2$ be defined by
$w =(q,1-c).$  Then the triangle
$T=\conv(r, w, e_2)$  is  regular: this
is a consequence of the unimodularity property
of any two neighbors in a Farey sequence.  The denominators of
  $r, w, e_2$ are    $d,1,1,$
respectively.
Letting similarly the triangle  $T'=\conv(r', w', e_2)$
   be defined with reference to  $r',$ it follows that
 the denominators of the vertices  $r', w', e_2$
 of   $T'$  are again  
 $d,1,1$.
    By { Lemma \ref{lemma:maps}}
  there is a map 
 $\gamma\in\mathcal G_2$  satisfying 
 $\gamma(r)=r', \,\, \gamma(w)=w',  \,\,
 \gamma(e_2)=e_2.$
Letting $\delta\in \mathcal G_2$ be defined
by composition as  $\delta=\beta^{-1}\gamma\alpha,$  
we conclude that  $\delta(x)=x'$, as desired.
 \end{proof}
 
\begin{remark} 
\label{remark:generalize}
 The generalization of
Proposition
\ref{proposition:rank1} to any dimension $n\geq 2$ 
is straightforward:
for a proof of the nontrivial direction,
for every   $x\in \mathbb Q^n$
we can  assume  $\den(x)=d>1$  (otherwise the result is trivial),
  and write  $x=(p/d,0,0,\ldots,0)$ for a
suitable  integer $0< p< d$  with $\gcd(p,d)=1$. 
Let  $F$   be the
subspace  of $\Rn$ spanned by $x$ and $e_2$, where
$e_2=(0,1,0,\ldots,0)$.
As  in the { proof of Proposition 
\ref{proposition:rank1},}
 for some integer vector  $w \in F$ 
there  is a regular 2-simplex  $T=\conv(x,e_2,w)$
with  $\den(w)=1.$
A routine construction now represents $T$ as
a face of some  regular $n$-simplex,
$U=\conv(x,e_2, w, v_3,\ldots,v_{n})$  in $\Rn$, for 
 suitable integer vectors $v_3,\ldots,v_{n}$.
 Similarly, given $x'\in\mathbb Q^n,$ 
 the  hypothesis  $G_{x'}=G_x$ means that
$\den(x')=\den(x)=d.$  Without loss of
generality we can write  $x'=(p'/d,0,0,\ldots,0)$,
and show that  also $x'$    is the vertex
of some regular  $n$-simplex  $U'$  in $\Rn$  whose remaining
vertices   are integer vectors.
By { Lemma \ref{lemma:maps}}
some  $\gamma\in\mathcal G_n$  satisfies
$\gamma(x)=x'$.
 
\smallskip
 The case $n=1$ will be dealt with in { Corollary
 \ref{corollary:onedim}.}
\end{remark}

 \section{Rank 2: More on regularity}
 The classification of orbits of points $x\in\R2$ with $\rank(G_x)=2$
 is  more laborious: the
 following example shows that
 in general,  $G_x$ alone  is not a complete classifier of $\orb(x)$:

\begin{example}
\label{example:nonbasta}
Let
$a=(1/5,\xi)$ and $b=(2/5,\xi)$, for some
irrational  $\xi\in[0,1]$.  Then
$G_a=G_b$ but for no
$\gamma\in \mathcal G_2$ we can have
$\gamma(a)=b.$  
\end{example}

\begin{proof}
For otherwise (absurdum hypothesis),
we can write
$$
\left(
\begin{matrix}
x & y &z  \\[0.1cm]
u&v&w\\[0.1cm]
0&0&1
 \end{matrix}
     \right)\left(
\begin{matrix}
1/5 \\[0.1cm]
\xi \\[0.1cm]
1
 \end{matrix}
     \right) =
     \left(
\begin{matrix}
2/5 \\[0.1cm]
\xi \\[0.1cm]
1
 \end{matrix}
     \right)
$$
for some unimodular integer $3\times 3$
matrix  $U$.
Since  $\xi$  is irrational, from 
$
x/5+y\xi+z=2/5$
we get $y=0 \mbox{ and } x=2-5z.
$
Further, from  
$
u/5+v\xi+w=\xi $ we get $ v=1 \mbox{ and } w=-u/5.
$
In conclusion, the unimodularity of $U$ yields
$$
\det\left(
 \begin{matrix}
2-5z & 0 \\[0.1cm]
u &1
 \end{matrix}
     \right)
     =\pm 1, 
$$
 {whence}
$2-5z=\pm 1$, which is impossible.
\end{proof}

  { In Definition \ref{definition:cx}}
we will introduce a second  invariant
$c_x\in \mathbb Z$  and show that the pair  $(G_x,c_x)$
  provides a complete classification of $\mathcal G_2$-orbits.
 
 \medskip
For later purposes we record here the following 
easy lemma:

 \begin{lemma}
 \label{lemma:subbasis}
 For $n=1,2,\ldots,$  let $F$ be an $e$-dimensional
  rational affine space in $\Rn$ with
 $0\leq e <n$.  Let 
$\conv(v_0,\ldots,v_e)$  and $\conv(w_0,\ldots,w_e)$ be two regular
$e$-simplexes lying on $F$.
Suppose
$\conv(v_0,\ldots,v_e,v_{e+1},\ldots,v_{n})$  is a regular
$n$-simplex, for suitable  points $v_{e+1},\ldots,v_{n}\in \Rn$.
 Then also
$\conv(w_0,\ldots,w_e,v_{e+1},\ldots,v_{n})$ is a regular
$n$-simplex. 
\end{lemma}


The following characterization of regularity will find repeated use
throughout the rest of this paper:

\begin{lemma}
\label{lemma:regular-simplex}
Given   integers   $n\geq 1$ and 
and   $\,\,0\leq m \leq n\,\,$,
let    $\,\,\,T = \conv(v_0, \ldots, v_m)$
be an $m$-simplex in $\Rn$. 
Then the following conditions are equivalent:

\begin{itemize}
    \item[(i)] $T$ is regular.

%
%

 \smallskip
   \item[(ii)]
The half-open parallelepiped    $P_{T}=\{x\in \mathbb R^{n+1}\mid x=
\sum_{i=0}^{m} \lambda_i \tilde{v}_{i},\,\,\,
0\leq \lambda_i<1\}$
    contains no nonzero integer point.


 \smallskip 
       \item[(iii)] For every   nonempty  face
    $F$ of $T$ and every rational 
    point $r$ lying in the relative interior of $F$,
    the denominator of $r$ is $\geq$ the
    sum of the denominators of the vertices of $F$.
       
    \end{itemize}
\end{lemma}

\begin{proof}

(i)$\Rightarrow$(ii) Trivial.
(ii)$\Rightarrow$(i)  If $T$ is not regular  
some point $q$  of the abelian group
 $\mathbb Z^{n+1} \cap
              ( \mathbb R \,\tilde v_0+\dots+
       \mathbb R\,\tilde v_m) $ is not expressible
       as a linear combination of   $\tilde v_0,\ldots,\tilde v_m$  
       with integer coefficients.  From $q$ one easily obtains
       by translation a nonzero integer point $q'$ lying in $P_T$.

\smallskip
For the proof of (iii)$\Leftrightarrow$(i),  supposing
$W=\conv(z_1,\ldots,z_{k})$ is a nonempty face of $T$
we will write $g_i=\tilde z_i \,\,\,(i=1,\dots,k)$ for the
primitive generating vectors of the cone  $W^\uparrow=
\pos[g_1,\ldots,g_k]$,
and  $P_W$  for the half-open parallelepiped associated to $W$.
We further let  $\Sigma_W$ denote the sum of the denominators of the
vertices of $W$.
In more detail, letting
$g_{i,n+1}$ denote the
$(n+1)$th coordinate of   $g_i$ of
$W^{\uparrow}$,  
$\Sigma_W  =  g_{1,n+1}+g_{2,n+1}+\dots+g_{k,n+1}.$

 \medskip
(iii)$\Rightarrow$(i)
If $T$ is not regular, by
(ii)$\Rightarrow$(i)
    the half-open parallelepiped 
    $P_T$ contains a nonzero integer
 point  $u = (u_1,\ldots,u_{n+1}) \in \mathbb Z^{n+1}$. 
Without loss of generality, $u$ is
 primitive.  Let $a$ be the uniquely determined
 rational point  such that  $u =\tilde a.$
It follows that $a\in T$ and  $\den(a)=u_{n+1}=$
$(n+1)$th coordinate of $u$.
Let $S$ be the smallest
face of $T$ such that
$a$ belongs to $S$.
Then $a$ lies in the relative
interior of $S$,   and
$u$  
belongs to the half-open parallelepiped 
$P_S.$  Therefore,    
$\den(a)  < \Sigma_S$ and condition (iii) fails.   

\smallskip
(i)$\Rightarrow$(iii)
Suppose $T$ is  regular,
 $F= \conv(w_0,\ldots,w_l)$ is a nonempty face of $T$,
 and $r$  is a rational point in the relative interior of $F$.
 Then $F$ is regular,  and the
 homogeneous correspondent $\tilde r$ of $r$
 is a primitive integer vector lying
in the relative interior of the cone
    $F^{\uparrow}$.   For suitable integers   $n_i\geq 1$ we can write 
$\tilde r=n_0\tilde w_0+\dots+n_l\tilde w_l,$  whence the
$(n+1)$th coordinate
$\tilde{r}_{n+1}=\den(r)$  of 
$\tilde r$ satisfies the  inequality  
$\tilde{r}_{n+1}\geq \Sigma_F$.
\end{proof}

\section{Rank 2: the companion of a rational point in $[0,1]$}
\begin{definition}
\label{definition:companion}
Let $x = p/d\in[0,1]$,    where the integers
$d=1,2,\ldots$, and $p$  are relatively prime,
in symbols,  $\gcd(p,d)=1.$
  Then the
  {\it companion   $\comp(x)$  of}  $x$    is defined
  as follows:  $\comp(0)=1$,  $\comp(1)=0,$ $\comp(1/2)=1$;
 for $d=3,4,\ldots,$   $\comp(p/d)=q/c,$  where $q/c$
is the neighbor (= contiguous element) of  $x$ in $\mathfrak F_d$
having  smaller denominator than the other neighbor.
 
 \smallskip
  The well known unimodularity property  $qd-pc=\pm 1$
  of any two neighbors  $p/d$ and $q/c$  in a Farey sequence,    
shows that  $\comp(p/d)$  is unambiguously defined. 
\end{definition}

  \begin{lemma}
  \label{lemma:preconverse}
  Adopt the above notation:
  \begin{itemize}
\item[(i)]  For
$d=1,2,3,4$,  and $x=p/d\in [0,1]$,  $\den(\comp(x))=1.$

\item[(ii)] For every $x\in \mathbb Q\cap[0,1],$
 $\den(x)$ \,\,and\,\,  $\den(\comp(x))$ are relatively prime.

\item[(iii)]   
$
\mbox{For  all
$\,\,d\geq 3$, }\, \den(\comp(p/d))<d/2.
$
\end{itemize}
\end{lemma}

\begin{proof}  We only prove (iii). Let $q/c$ denote the
companion of $p/d.$
Let  $a/b$   be the other neighbor   of  $p/d$ in $\mathfrak F_d.$
Both segments $\conv(a/b,p/d)$  and
$\conv(q/c,p/d)$ are regular.  
The denominators  of both
neighbors are  smaller than $d$.
{ By Lemma \ref{lemma:regular-simplex},}
$a/b$  and $q/c$  are neighbors in
$\mathfrak F_{\max(b,c)}$,  the segment
$\conv(a/b,q/c)$ is regular, and  $d=b+c$. 
(The rational $p/d$  is known as the {\it Farey mediant}  of 
$a/b$ and $q/c$.)
Since  $d\geq 3$, then $c<b$ by definition of companion,     
and hence  $c<d/2.$
\end{proof}

Conversely,   we have:

\begin{lemma}
\label{lemma:companion}
Let  $d$ and $c$ be integers  $\geq 1$.
Suppose
\begin{itemize}
\item[(i)] either  $d=1,2,3,4$ and $c=1$,

\item[(ii)] or  $d\geq 5$ and $c<d/2$,  
with $\gcd(d,c)=1$.
\end{itemize}

 Then the  Farey sequence
 $\mathfrak F_d$  contains a pair of  rationals
 $p/d$ and $q/c$  such that
 $q/c=\comp(p/d).$ 
\end{lemma}

\begin{proof}  In case  (i) the conclusion is trivial.
In case (ii),  let $C_0=\pos[(0,1), (1,1)]$ be the positive hull
in $\R2$  of the two integer vectors $(0,1)=u_0$ and $(1,1)=v_0$.
There is a unique sequence  of cones 
$$
C_0=\pos[u_0,v_0]\supseteq C_1=\pos[u_1,v_1]
\supseteq C_2=\pos[u_2,v_2]\supseteq \cdots \supseteq C_t=\pos[u_t,v_t]
$$
 all containing the point 
$(c,d)$  in their interior, and with the additional property that
$v_{i+1}=u_i+v_i$  and  $u_{i+1}\in\{u_i,v_i\}.$   
Each set  $\{u_i, v_i\}$ is a basis of  the free abelian group
$\mathbb Z^2$.  
Let $(u_{t1},u_{t2})$ be the coordinates of the
integer vector  $u_t$. Let  similarly $v_t=(v_{t1},v_{t2})$.
Let   $I_i$ be the  
intersection of $C_i$  with the
line   $y=1$,   and denote by $J_i$ 
the projection of  $I_i$ onto the $x$-axis.
We then obtain a decreasing sequence of regular 1-simplexes
$$
J_0=\conv(0,1)\supseteq J_1=\conv\left(\frac{u_{11}}{u_{12}},
\frac{v_{11}}{v_{12}}\right)\supseteq J_2=\conv\left(\frac{u_{21}}{u_{22}},
\frac{v_{21}}{v_{22}}\right)
\supseteq \cdots
$$
By { Lemma \ref{lemma:regular-simplex} (i)$\Rightarrow$(iii)},
 the rational point $c/d$  cannot lie in the interior
 of infinitely many $J_i$'s.
 Thus  this  sequence terminates after a certain number
 $t$ of steps  with the segment  $J_t=\conv({u_{t1}}/{u_{t2}},
{v_{t1}}/{v_{t2}})$,   in such a way that
  $c/d$  is the
Farey mediant of the vertices of  $J_t$. 
We then have 
\begin{equation}
\label{equation:uno}
 (c,d)=  (u_{t1},u_{t2})+(v_{t1},v_{t2}).
\end{equation}
Observe that  $u_{t1}<u_{t2}$ and  $v_{t1}<v_{t2}$.
The $2\times 2$ matrix whose columns are
the vectors  $u_t$  and $v_t$ is unimodular.
Upon defining 
$$
 q=u_{t1} 
\,\,\mbox{ and }\,\,  p=u_{t2}\,,$$
it follows that
  $qd-pc=\pm1,$  meaning that the 1-simplex
$J=\conv(p/d,q/c)$   is regular.
As a consequence, 
$\gcd(q,c)=\gcd(p,d)=1.$
By (\ref{equation:uno}),
$q<c$ and $p<d$, whence both $p/d$ and $q/c$
belong to $\mathfrak F_d$.

To see  that   
$q/c$ is a neighbor of $p/d$ in $\mathfrak F_d$,
arguing by way of contradiction
suppose some rational
$r  \in \mathfrak F_d$ lies in the interior of $J$.
  { Lemma \ref{lemma:regular-simplex}
(i)$\Rightarrow$(iii)}  then
yields  $\den(r) \geq d+c$, against the definition of
$\mathfrak F_d$. 

 To verify that  $q/c$  is the companion of $p/d,$  
let $a/b$ be the other
neighbor  of
$p/d$ in $\mathfrak F_d$.  
Since the segment  
with vertices $a/b$ and $p/d$ is regular
and by hypothesis $d>1,$ then $b<d$.
Since $c<d$  then
  $\max(c,b)< d$.
One more application of  {  Lemma  \ref{lemma:regular-simplex}}
now shows that 
$q/c$  and $a/b$ are contiguous elements of
 $\mathfrak F_{\max(c,b)}$, the
  segment   $\conv(q/c,a/b)$   is
regular, and   
$d=c+b$, i.e.,  $p/d$ is the Farey mediant of
its neighbors. 
Since by hypothesis  $c<d/2$ then $c<b.$
Having thus shown   
  that $q/c=\comp(p/d)$,
  the proof is complete.
\end{proof}

\section{Rank 2: The complete classifier
  $(G_x, c_x)$}

As shown by 
{ Example \ref{example:nonbasta}},
when  $\rank(G_x)=2$  the group $G_x$  does not suffice
in general to classify $\orb(x)$.
In { Definition \ref{definition:cx}}
we  will define a numerical invariant 
which,  jointly with the group $G_x$
completely classifies $\orb(x).$ 
As a preliminary  step  we introduce
the following notation and terminology: 

 \begin{definition}
\label{definition:d-and-F-of-x}
For any fixed   $n=1, 2,3,\ldots$ and  
 rational affine space  $\emptyset \not=F\subseteq \mathbb R^n$,
we let $
d_F =  \min\{q \in \mathbb Z
 \mid q \mbox{\,\,\,is the denominator of a rational point of} \,\,\,F\}.
 $
For any  $x\in \Rn$ we define  
$
F_x=\bigcap \{A\subseteq \Rn \mid A \mbox{ is a rational 
hyperplane in  $\Rn$ and $x\in A$}\}.
$
\end{definition}

\begin{lemma}
\label{lemma:wlog-isobunch}  
For  $n=1, 2,3,\ldots$ and $e=0,\ldots,n,$
let  $F$ be  an $e$-dimensional rational affine space in 
 $\Rn$.
Let  $d=d_F$.
\begin{itemize}
\item[(i)] There are rational points
 $v_0,\ldots,v_e\in F$, all with denominator $d$,
 such that $\conv(v_0,\ldots,v_e)$ is a regular
 $e$-simplex. 
 
\item[(ii)] For any rational point $y\in F$
 there is $h=1,2,\ldots$  such that $\den(y)$$=hd.$
 \end{itemize}
\end{lemma}

\begin{proof}
(i)  Starting from any rational
$e$-simplex  $R\subseteq F$  and arguing as in
\cite[Proposition 1]{mun-dcds}, one easily obtains
a regular $e$-simplex   
$\conv(u_0,\ldots,u_e)\subseteq R.$

 For each $i=0,\ldots,e$
the integer vector $\tilde u_i\in \mathbb Z^{n+1}$ 
will be  more conveniently
denoted  $w_{0,i}$.
Let $F^*$  be the linear span in
$\mathbb R^{n+1}$  of the primitive integer vectors
$w_{0,0},\ldots,w_{0,e},$
$
F^*=\mathbb Rw_{0,0}+\dots+\mathbb Rw_{0,e}.
$
The regularity of 
$\conv(u_0,\ldots,u_e)$  means that the
 set  $B_0= \{w_{0,0},\ldots,w_{0,e}\}$ is a basis
of the free abelian group $\mathbb Z^{n+1}\cap F^*$. 
If the {\it heights}  (i.e., the last coordinates) of 
$w_{0,0},\ldots,w_{0,e}$
are all equal to $d$ we are done.
Suppose not all heights of the vectors in $B_0$ are equal to
 $d$.  We  will construct a finite sequence $B_0, B_1,\ldots$
 of bases of 
   $\mathbb Z^{n+1}\cap F^*$, 
 and  finally  obtain a  basis  $B_{t^*}$  yielding the desired
regular $e$-simplex $\conv(v_0,\ldots,v_e)$.
 
 The step leading from $B_0$  to $B_1$ 
 begins with the observation that   all
 heights   of $w_{0,0},\ldots,w_{0,e}$  are
$\geq d$,   and yet
it is impossible that they are all
equal to the same integer
$l>d$:  for otherwise,
a (necessarily primitive) vector  in $\mathbb Z^{n+1}\cap F^*$
  of height $d$,
as given by the definition of $d$, 
 could not arise
as a linear combination of 
$w_{0,0},\ldots,w_{0,e}$ with
integer coefficients, and $B_0$  would not be a basis
of  $\mathbb Z^{n+1}\cap F^*$. 
Since not all vectors
of $B_0$  have the same height,
 we  pick in $B_0$
a vector $w_{0,i}$  of largest height, a vector 
$w_{0,j}$  of smallest height, and replace
$w_{0,i}$  by the new integer vector
$w_{0,i}-w_{0,j}$ leaving the rest unchanged.
  We get a new basis
$B_1=\{w_{1,0},\ldots,w_{1,e}\}$  of  
$\mathbb Z^{n+1}\cap F^*$.
The height of each  element of $B_1$ is still $\geq d.$
The  sum of the heights of the elements  of $B_{1}$ is strictly smaller
than the sum of the heights of the elements  of $B_{0}$.  
The positive hull of these vectors yields an
$(e+1)$-dimensional regular cone $C_1$,
from which we get a regular $e$-simplex
$\conv(w_{1,0}^\downarrow,\ldots,w_{1,e}^\downarrow)
\subseteq \Rn$
by letting  $w_{1,i}^\downarrow$ be the only
rational point in $\Rn$ such that $w_{1,i}$
is the homogeneous correspondent of $w_{1,i}^\downarrow$.

Proceeding inductively, suppose we have a basis
$B_t=\{w_{t,0},\ldots,w_{t,e}\}$ of $\mathbb Z^{n+1}\cap F^*$, 
whose positive hull is the cone $C_t\subseteq \mathbb R^{n+1}$, 
and its associated
regular $e$-simplex is   
$\conv(w_{t,0}^\downarrow,\ldots,w_{t,e}^\downarrow)\subseteq \Rn$.
If the heights of all the vectors of $B_t$ are equal to $d$ we are done.
Otherwise, as above, 
we are allowed to  pick in $B_t$ a top vector  $w_{t,i}$,
a bottom vector $w_{t,j}$, and replace $w_{t,i}$
by $w_{t,i}-w_{t,j}$   thus obtaining
a new basis $B_{t+1}=\{w_{t+1,0},\ldots,w_{t+1,e}\}$.
{}From $B_{t+1}$ we get a new
 regular cone  $C_{t+1}\subseteq \mathbb R^{n+1}$
and a new regular 
$e$-simplex
$\conv(w_{t+1,0}^\downarrow,\ldots,w_{t+1,e}^\downarrow)
\subseteq \Rn$
by the same procedure of the first step.
The height of each  element of $B_{t+1}$ is still $\geq d.$
Since the 
sum of the heights of the elements of    
$B_{t+1}$ is strictly smaller
than the sum of the heights of the elements  of $B_{t}$,
this sequence terminates with a  basis  $B_{t^*}$  of
$\mathbb Z^{n+1}\cap F^*$   whose elements
all have the  same height  $d$.  From
$B_{t^*}$ we obtain the regular cone $C_{t^*}$ 
and the desired regular $e$-simplex $\conv(v_0,\ldots,v_e)$.

  (ii) is a trivial consequence of (i): for, 
 the primitive vector
$\tilde y$ is a linear combination of the
$\tilde v_i$ with integer coefficients, and
$1\leq \den(y)\in \mathbb Z$.
\end{proof}

 The following characterization of $d_{F_x}$
 will find repeated use in the sequel:
 
\begin{lemma}
\label{lemma:missing}
With reference to {  Definition
\ref{definition:d-and-F-of-x},} for 
every  $x\in \R2$ with $\rank(G_x)=2$
let us write for short   $d$ instead of $d_{F_x}$.
Let us   display  $F_x$ as the
  rational line  $\{(x_1,x_2)\in\R2\mid a_1x_1+a_2x_2+a_3=0\}$,
for relatively prime integers  $a_1,a_2,a_3$.
Noting that the primitive integer  vector $(a_1,a_2,a_3)\in \mathbb Z^3$
is uniquely determined up to scalar multiplication by $\pm 1$ 
 we then have:  
\begin{itemize}

\item[(i)]    $d=\gcd(a_1,a_2).$

\smallskip
\item[(ii)] $d$ is the  denominator of the smallest nonzero 
positive rational in $G_x.$
\end{itemize}

\end{lemma}

\begin{proof}
(i)  Let  us use the abbreviations  $D=\gcd(a_1,a_2)$, 
$  \,\,a'_1=a_1/D,
\,\,a'_2=a_2/D$. 
Then
$ \gcd(D,a_3)= \gcd( a'_1, a'_2)= 1$ and for suitable  
  integers $\alpha, \beta$
we can write
$$\alpha a'_1+\beta a'_2 =-1,  \,\,\, \gcd(\alpha,\beta)=1.$$
Thus 
$ 
a_1(\alpha a_3/D)+a_2(\beta a_3/D)=-a_3$, and 
the rational point  
 $({\alpha a_3}/{D},{\beta a_3}/{D})$ belongs to  $F_x.$ 
 {}From $\gcd(D,a_3)= \gcd(\alpha,\beta)=1$ it follows
 that $\den({\alpha a_3}/{D},{\beta a_3}/{D})=D.$
 
 To conclude the proof that $d=D,$
 we {\it claim}  that no rational point
  $z\in \R2$  with $\den(z)<D$
lies  in $F_x.$ 
  For otherwise, for suitable integers $m,n,\delta,$
   let $z=(m/\delta,n/\delta)$
  be a counterexample  (absurdum hypothesis).
{}From 
$a_1m/\delta+a_2n/\delta=-a_3$ we get 
$a_1m +a_2n =-a_3\delta.$  On the other hand,
$a_1m+a_2n=kD$ for some integer $k$.
It follows that
 $a_3\delta=kD.$  Since   $\delta<D$  
 some prime divisor of $D$ must also be a divisor of $a_3$,
 which contradicts   $\gcd(D,a_3)=1$.
Our claim is settled, and hence
$d=D=\gcd(a_1,a_2)$,
as required to complete the proof of  (i).

\medskip
(ii)  Let  $H=\{\tilde y\in\mathbb Z^3\mid y\in F_x\cap \mathbb Q\}$
be the set of homogeneous correspondents of all rational points
of $F_x.$  Then every point $z\in F_x$ of denominator $d$ corresponds
to an integer point $\tilde z\in H$ of height $d$.
One then sees that  points of denominator $d$ lying on
the line  
$F_x$  are equidistant. Each
segment in $F_x$  joining any two consecutive
such points is regular: this follows
from   {  Lemma
 \ref{lemma:regular-simplex}(iii)$\Rightarrow$(i) 
in combination with Lemma
 \ref{lemma:wlog-isobunch}}.
The point $x$
lies in the convex hull of precisely two points $a,b\in F_x$
of denominator $d$.

Pick a rational point $s\not\in F_x$  such that the triangle $T=\conv(a,b,s) $  is regular.
Let $c=\den(s)$.  The regularity of $T$ is to the effect
that $\gcd(d,c)=1$.  Then
  { a  routine variant of the argument in  Lemma 
\ref{lemma:companion}}  
yields  rationals $q/c$ and $p/d$  (written in their
lowest terms)  which are
neighbors  in the Farey sequence  $\mathfrak F_{\max(c,d)}.$
The triangle $U=\conv((1/d,0),(0,p/d),(0,q/c))$ is regular.
By { Lemma \ref{lemma:maps}}
there is a map $\gamma \in \mathcal G_2$ such that
$\gamma(T)=U$.
In particular  $\gamma(x)=\xi=(\xi_1,\xi_2)$
for suitable irrationals $\xi_1$ and $\xi_2$.
It follows that   $G_x=G_\xi$  is generated by
$\xi_1,\xi_2,1.$
The point $\xi$ lies  on the line $F_\xi$.
Thus
$\xi_2
=p(1/d-\xi_1).$
We see that $p/d\in G_x.$   Since $1\in G_x$, from
$\gcd(p,d)=1$  we obtain  $1/d\in G_x.$

To conclude the proof, by way of contradiction  let
$u/v$ be a rational in $G_x$  with  
\begin{equation}
\label{equation:francia}
0 <u/v<1/d.
\end{equation}
{}From  $\gcd(u,v)=1$ and  
 $1\in G_x$  we obtain   $1/v\in G_x,$  whence
  $1/\lcm(d,v) \in G_x,$  where \, lcm\, denotes least
  common multiple.    Upon writing  
  $\lcm(v,d)=kd$  for some integer  $k$, since 
  $\rank(G_x)=2$ then
 $\{1/d,\xi_1\}$ is a {\it free} generating set
of  $G_x$,
and there are integers  $m,n$  such that  $m\xi_1+n/d=1/(kd),$
whence  $nk=1$,    $\,\,\,n=k=1$,  and $\,\,v=d$.
This contradicts (\ref{equation:francia}).
\end{proof}

We are now ready to introduce the numerical invariant
  which, in combination  with  $G(x)$,  yields a
complete invariant for $\orb(x)$. 

\begin{definition}
\label{definition:cx}
For   any  $n=1,2,\ldots$ and
$e$-dimensional
rational affine space $F$  in $\Rn$
($e=0,1,\ldots, n$),  we define
 $c_F$  as the smallest possible denominator
of a vertex of a regular $n$-simplex  $\conv(v_0,\ldots,v_n)
\subseteq \Rn$
with $v_0,\ldots,v_e\in F.$
\end{definition}
For every $x\in\Rn$  we write
$c_x$ as an abbreviation of $c_{F_x}.$

 \smallskip
The two points $a,b$  of { Example \ref{example:nonbasta}}
have  $c_a=1$  and $c_b=2.$

\begin{lemma}
\label{lemma:cx}
Let $n =2,3,\ldots$ and $x\in \mathbb R^n.$
\begin{itemize}
\item[(i)]  If $\rank(G_x)=n+1$, i.e.,  if
$F_x$ is $n$-dimensional,  then $ c_x=1.$

 \item[(ii)] If  $\rank(G_x)=1$, i.e., if
$F_x$ is $0$-dimensional,  then $c_x=1.$

 \item[(iii)]  Suppose  $F_x$  is  $e$-dimensional,
 with $0 <  e < n$,  meaning that
 $\rank(G_x)\not\in\{1,n+1\}$.
   Write for short  $d$ instead of $d_{F_x}$.
 If $d\in \{1,2\}$ then $c_x=1.$
  If   $d\geq 3$ then $c_x<d/2$ and $\gcd(d,c_x)=1.$
\end{itemize}
\end{lemma}

\begin{proof}
(i)  The hypothesis  means that $x$ is totally irrational. 
Then $F_x=\Rn.$  It is easy to construct  a
regular simplex  $T$ in $\Rn$  such that the denominator
of each vertex of $T$ is equal to 1.   By definition,
$c_x=1.$

(ii) {  By  Remark \ref{remark:generalize}.}

(iii) 
Let  $V=\conv(v_0,\ldots,v_n)\subseteq \Rn$   be
a regular $n$-simplex
with  $v_0,\ldots,v_e\in F_x$,
and with  a vertex $v_i$  satisfying 
$\den(v_i)=c_x.$ 

\smallskip
\noindent{\it Claim 1.}  
There is a regular 
$n$-simplex  $W=\conv(w_0,\ldots,w_n)
\subseteq \Rn$
whose vertices $w_0,\ldots,w_e$ lie in $F_x$,
and with   
$\den(w_n)=c_x.$

As a matter of fact,  if  $\den(v_j)=c_x$
  for some $j=e+1,\ldots,n$, then
  without loss of generality  $j=n,$ and we are done.
Otherwise,  
 for all  $j=e+1,\ldots,n$ we have the inequality 
   $\,\,\,c_x<\den(v_j)$,  and we can  write   $\den(v_0)=c_x$ 
 without loss of generality.
 Passing to homogeneous correspondents in $\mathbb R^{n+1}$,
 for some  integer $k>0$
the height    (= last coordinate)  $l$ of the   vector  
$r_n=\tilde v_n-k\tilde v_0\in \mathbb R^{n+1}$ will  satisfy  $1\leq l \leq$
the height of $\tilde v_0$.   Thus by definition of
$c_x$,  both  $r_n$  and $\tilde v_0$  have the
same height $c_x.$  Let  $z_n \in \Rn$ be the rational point
 whose homogeneous
correspondent  $\tilde z_n$   equals   $r_n$.
Then   $\den(z_n)=c_x=$
  height of $r_n$.  The regular $n$-simplex
$W=\conv(v_0,\ldots,v_{n-1},z_n)$ satisfies the
conditions of our claim.  

\smallskip
An application of 
  {  Lemmas \ref{lemma:wlog-isobunch} and 
\ref{lemma:subbasis}}
now yields

\medskip
\noindent{\it Claim 2.}
There is a regular 
$n$-simplex  $Y=\conv(y_0,\ldots,y_n)
\subseteq \Rn$
with  $y_0,\ldots,y_e\in F_x$,
satisfying the identities 
$\den(y_n)=c_x$ and $ \den(y_0)=\dots=\den(y_e)=d.$

\medskip

We  conclude the proof arguing by cases: 

\medskip

   If $d=1$ we are done, because $1\leq c_x\leq d$.

\smallskip

If  $d=2$ we {\it claim}  that  $c_x=1$.
If not (absurdum hypothesis),  $c_x=2$  and  not all of
$\den(y_{e+1}),\ldots,\den(y_n)$ can be even,
because the regularity of the $n$-simplex 
$Y=\conv(y_0,\ldots,y_n)$  entails the unimodularity
of the matrix  whose columns are $\tilde y_0,\ldots,\tilde y_n$.
Thus for some $i=e+1,\ldots,n,\,\,$  
$\den(y_i)$ is odd and $\geq 3$.
For  some  integer  $k\geq 1$, 
the height of the integer vector
$l_i =\tilde y_i-k\tilde y_0\in \mathbb R^{n+1}$ is
 equal to 1.   Let  the rational point
$u_i\in \Rn$  be defined by 
$\tilde u_i=l_i.$  Replacing in $Y$
the vertex   
$y_i$   by  $u_i$, the resulting
 $n$-simplex
$\conv(y_0,\ldots,y_{i-1},u_i,y_{i+1}\ldots,y_n)$
is regular.  Since  $\den(u_i)=1$  then
$c_x=1,$  a contradiction with our absurdum hypothesis
$c_x=2$.  
We conclude that 
$c_x=1$, as required to settle our claim.

\smallskip
If  $d\geq 3,$  for every  $l=e+1,\ldots,n-1$  such that $\den(y_l)>c_x$
there is an integer  $k_l\geq 1$ such that
the height $h$ of the 
vector  $s_l=\tilde y_l-k_l\tilde y_n\in \mathbb Z^{n+1}$
satisfies  $1\leq h \leq c_x$,  whence  $h=c_x.$
Letting  the rational point  $u_l$  be defined by
$\tilde u_l= s_l$,  replace  vertex $y_l\in Y$  by
$u_l.$ 
We then obtain  a new regular
$n$-simplex   with at least
two vertices not in $F_x$, both having denominator
equal to $c_x$.
This process terminates with a regular $n$-simplex
$Z=\conv(y_0,\ldots,y_e, u_{e+1},\ldots,u_{n-1},y_n)$
whose first  $e+1$    vertices lie in $F_x$
and  have denominator $d$, 
 and whose remaining vertices have denominator $c_x$
 and lie outside $F_x.$ 
The regularity of $Z$ entails
that $c_x$  and $d$ are relatively prime. Since  $d>1$  then  
$c_x<d.$  Since  $c_x$  is minimal, then
$c_x\leq d/2$. Actually,  $c_x<d/2$.
For otherwise,  $d$ is even,
 $d\geq 4$,  and
$\gcd(c_x,d)=c_x\geq 2,$  contradicting 
 the regularity of $Z$.
 \end{proof}

\section{Rank 2: the classification theorem} 
Our classification of $\mathcal G_2$-orbits is completed
as follows:

\begin{theorem}
\label{theorem:rank2}
If  $x\in \R2$ and  $\rank(G_x)=2$
then  $\orb(x)=\{x'\in\R2\mid (G_x,c_x)=(G_{x'},c_{x'})\}$. 
\end{theorem}

\begin{proof}
The  rational affine subspace
$F_x\subseteq \R2$  is the only rational  line in $\R2$ 
passing through 
$x$.   
 We first prove 
$\orb(x)\subseteq \{x'\in\R2\mid G_x=G_{x'} \mbox{ and }
c_x=c_{x'}\}$.  
If $x$ and $x'$ satisfy $\delta(x)=x'$ 
for some $\delta\in \mathcal G_2$,
then we immediately see that  $G_{x}=G_{x'}.$  
Further, the map 
 $\gamma$ sends $F_x$ one-one onto $F_{x'}$. Trivially,
 $\gamma$ preserves denominators of rational
points.  { Lemma \ref{lemma:regular-simplex}(i)$\Leftrightarrow$(iii)}  
shows that $\gamma$ also induces a one-one correspondence
between regular simplexes. Thus $c_x=c_{x'}$.

\smallskip
  For the converse inclusion
$\orb(x)\supseteq \{x'\in\R2\mid G_x=G_{x'} \mbox{ and }
c_x=c_{x'}\}$, let us suppose
 $G_x=G_{x'} \mbox{ and }
c_x=c_{x'}$,   looking for an affine transformation
  $\delta\in\mathcal G_2$ such that  $\delta(x)=x'$.
For suitable integers 
$a_1,a_2,a_3$ with  $\gcd(a_1,a_2,a_3)=1$,  and
 $a_1$ and $a_2$ not both zero, we can write
$
F_x  =  \{(y_1,y_2) \in \R2 \mid a_1y_1+a_2y_2+a_3=0\}.
$
By  { Lemma \ref{lemma:missing}}, 
$d_{F_x}$ coincides with  the  denominator of the smallest
positive nonzero rational point of
 $G_x.$  Let us write  $d$ instead of $d_{F_x}$. 
The points of denominator $d$ lying on 
the line $F_x$   are at equal distance. Each
segment in $F_x$  joining any two consecutive
such points is regular  { by Lemmas
 \ref{lemma:regular-simplex}(iii)$\Rightarrow$(i)
 and \ref{lemma:wlog-isobunch}(ii)}.
 The point $x$
lies in the convex hull of precisely two points $a,b\in F_x$
of denominator $d$.

{ The same
  construction in the proof of Lemma \ref{lemma:cx}(iii)}
yields a rational  point 
$s\not\in F_x$  such that the triangle $T=\conv(a,b,s) $  is regular,
and $\den(s)=c_x.$ Let us write $c$ instead of $c_x.$ 

The denominators of 
$a,b,s$ are equal to   $d,d,c$ respectively.
As a consequence of the regularity of $T$,  $\gcd(c,d)=1.$
By   { Lemmas \ref{lemma:cx}}
the  integers $c$  and $d$ satisfy the hypotheses of
{ Lemma \ref{lemma:companion}}, whence
we obtain integers  $p<d$  and  $q<c$ such
that  $q/c=\comp(p/d)$ and
\begin{equation}
\label{equation:unimod}
qd-pc=\pm1.
\end{equation}
 
Let the triangle  $\bar T=\conv(\bar a,\bar b, \bar s)$ be defined by 
$
\bar a =(1/d,0),  \bar b=(0,p/d), \bar s=(0,q/c).
$
Then $\bar T$  is regular and the denominators of
its vertices are respectively equal to  $d,d,c$.
{ By Lemma  \ref{lemma:maps}}  some $\alpha \in \mathcal G_2$ maps
$a,b,s$  to $\bar a,\bar b, \bar s$.  For suitable
irrational numbers  $\xi_1,\xi_2$ we can write
 $\alpha(x)={\bar x}=
(\xi_1,\xi_2)\in \conv(\bar a,\bar b)$.
Evidently,   $G_x=G_{\bar x}$, the latter being 
  the subgroup of $\mathbb R$ generated by $\xi_1,\xi_2,1$.
Since  ${\bar x}\in F_{\bar x} = \alpha(F_{x})$ then
\begin{equation}
\label{equation:due}
\xi_2=p(1/d-\xi_1). 
\end{equation}
Thus the  subgroup  $H$  of $\mathbb R$  generated
by $\xi_1$ and $1/d$ contains $G_{\bar x}$  as a subgroup, $G_{\bar x}\subseteq H.$
Conversely,  $H\subseteq G_{\bar x}$.
As a matter of fact, from 
 $\gcd(p,d)=1$ it  follows that 
$lp+md=1$ for suitable integers $l,m$.
Thus    
$l\cdot p/d+m\cdot 1=1/d$, and hence  
$1/d\in G_{\bar x}$,  because    
 $p/d  \in G_{\bar x}$,  by (\ref{equation:due}).  We have shown
 $
 G_x=G_{\bar x}=H = \mathbb Z\xi_1+\mathbb Z\frac{1}{d}.
 $
{}From the hypothesis  $\rank(G_x) =2$ it follows that
\begin{equation}
\label{equation:free}
\mbox{$ G_x$ is  freely generated by
 $\xi_1$ and $1/d$.}
 \end{equation}

We now consider the point  $x'$  and its associated
group $G_{x'}=G_x.$
By  {  Lemma \ref{lemma:missing}}
the  integer   $d'=d_{F_{x'}}$ 
coincides with  the  denominator of the smallest
positive nonzero rational point of $G_{x'}$.
Since  $G_{x'}=G_x$  then   $d'=d.$
By hypothesis,  $c_x=c_{x'}=c$.   
The point $x'$
lies in the convex hull of precisely two points $a',b'\in F_{x'}$
of denominator $d$.  As above, there
is a rational point    $s'\not\in  F_{x'}$  such that the triangle 
$T'=\conv(a',b',s') $  is regular,
and $\den(s')=c$.
Since the denominators of the vertices  $a',b',s'$
  are equal to     $d,d,c$,  some   $\beta
\in \mathcal G_2$ will map $T'$
onto {\it the same} triangle 
$\alpha(T)= \conv(\bar a,\bar b, \bar s)$.
In particular,  $\beta(x')={\bar x}' $
for some point  ${\bar x}'=(\xi'_1,\xi'_2)\in \conv(\bar a,\bar b)$.
It follows that
\begin{equation}
\label{equation:freeprime}
\mbox{$G_{x'}=G_x$ is freely generated by
 $\xi'_1$ and $1/d$.}
 \end{equation}
Recalling (\ref{equation:free}),
some  $2\times 2$  integer unimodular matrix  $U$
maps the column vector $(\xi_1,1/d)$ to the column vector 
 $(\xi'_1,1/d)$. Since $\xi_1$  is irrational,
  the bottom row of $U$ has the form  $(0,1)$.
 So, for some  $h=\pm 1$,
\begin{equation}
\label{equation:croce}
\left(
\begin{matrix}
h &k \\[0.1cm]
0&1
 \end{matrix}
     \right)\left(
\begin{matrix}
\xi_1 \\[0.1cm]
1/d
 \end{matrix}
     \right) =
     \left(
\begin{matrix}
\xi'_1 \\[0.1cm]
1/d
 \end{matrix}
     \right)
\end{equation}

\medskip
To get  the desired map $\delta\in  \mathcal G_2$  sending
$x$ to $x'$ we have only to exhibit  a  map
$\gamma\in \mathcal G_2$  with  $\gamma({\bar x})={\bar x}'$.

\bigskip
\noindent{\it Case 1:  $h=1$}.
Then  $k=0$, because  
   $0<\xi'_1<1/d$, and
     $0<\xi_1<1/d$.
It follows that ${\bar x}={\bar x}'$  and the identity map
on   $\R2$  provides the required $\gamma\in\mathcal G_2.$

\bigskip
\noindent{\it Case 2:  $h=-1$}.
Then {}from $\xi'_1=-\xi_1+k/d$  we immediately
obtain $k=1,$  because  
   $0<\xi'_1<1/d$, and
     $0<\xi_1<1/d$.
{}From (\ref{equation:croce}) have the identities 
\begin{equation}
\label{equation:primouno}
\xi'_1=-\xi_1+1/d,  \quad
\xi'_2=p\xi_1.
\end{equation}
Since  $\xi'_2$  lies in $F_x$,
\begin{equation}
\label{equation:primoduebis}
\xi'_2=p(1/d-\xi'_1).
\end{equation}

\medskip
To conclude the proof, recalling (\ref{equation:unimod}),
 let  
$u=cp-dq=\pm 1.$ 
Let    the integer matrix  $M$ be given by

$$
M= \left(
\begin{matrix}
 dq/u &  c/u &    -q/u   \\[0.3cm]
p-dpq/u & -dq/u &    pq/u  \\[0.3cm]
0&  0 &   1  
 \end{matrix}
     \right)
$$

\medskip
\noindent
Then  $M$ is unimodular and
 maps  $(\xi_1,\xi_2,1)$ to $(\xi'_1,\xi'_2,1)$.
To see this, one verifies the identity

\medskip
\noindent 
$$
M\left(
\begin{matrix}
\xi_1 \\[0.1cm]
p(1/d-\xi_1) \\[0.1cm]
1 
 \end{matrix}
     \right) =
     \left(
\begin{matrix}
-\xi_1+1/d \\[0.1cm]
p\xi_1\\[0.1cm]
1 
 \end{matrix}
     \right)
$$

\medskip
\noindent
using
(\ref{equation:due})-(\ref{equation:freeprime}) and
(\ref{equation:primouno})-(\ref{equation:primoduebis}).
{}From $M$ we immediately obtain
the desired $\gamma\in\mathcal G_2$
satisfying   $\gamma({\bar x})={\bar x}'.$  
\end{proof}

 \medskip

The following result states that the pair  $(G_x,c_x)$ is a complete invariant
for the orbit of $x\in \R2:$  

\begin{corollary} 
Two points $x,x'\in\R2$  have the same orbit
iff $G_x=G_{x'}$ and $c_x=c_{x'}$.
\end{corollary}

\begin{proof}
The  case $\rank(G_x)=2$ is taken care of by the { foregoing 
result. }  The cases  $\rank(G_x)\in\{1,3\}$ are dealt with
by {    Lemma
\ref{lemma:cx}(i)-(ii)}.  Here   $c_x$ is automatically
equal to 1, and its role in the classification of
$\orb(x)$  is immaterial.
\end{proof}


 \begin{corollary}
 \label{corollary:tutti}
 For  $x=(x_1,x_2)$ ranging over all points in $\R2$  let the subgroup
 $G_x$ of $\mathbb R$  be defined by
 $G_x=\mathbb Z+\mathbb Zx_1+\mathbb Zx_2.$
 Let $c_x$  be as in Definition \ref{definition:cx}.
 Then the possible classifiers  $(G_x,c_x)$
 of $\mathcal G_2$-orbits  are as follows: 
\begin{itemize}
\item[(i)] If  $\rank G_x\in\{1,3\}$ then automatically $c_x=1.$
\item[(ii)] If  $\rank G_x=2$,   
let  $d$ be the denominator of the smallest  nonzero
positive rational element of $G_x$.
If   $d\leq 4$  then $c_x=1.$
If  $d = 5,6,\ldots$  then any integer $c$   relatively
prime with $d$ and such that $0< c<d/2$  may occur as $c_x$. 
Thus precisely  $\max(1,\phi(d)/2)$ distinct orbits
are classified by the same group $G_x.$

\item[(iii)]
No other pair  $(G_x,c_x)$ may occur.  
\end{itemize}
\end{corollary}

\begin{proof} 
{ From Lemmas 
 \ref{lemma:preconverse}, \ref{lemma:companion}, \ref{lemma:cx}
 and Theorem \ref{theorem:rank2}.}
\end{proof}

The complete classification
of $\mathcal G_1$-orbits is now easily obtained as follows:

\begin{corollary}
\label{corollary:onedim}
For  $x$ ranging over all points in $\mathbb R$  let the subgroup
 $G_x$ of $\mathbb R$  be defined by
 $G_x=\mathbb Z+\mathbb Zx.$
 Let $c_x$  be as in Definition \ref{definition:cx}.
 Then the possible classifiers  $(G_x,c_x)$
  of $\mathcal G_1$-orbits   are as follows: 
\begin{itemize}
\item[(i)]
In case  $\rank(G_x)=2$  then  $c_x=1$. 
For  all $y\in \mathbb R,$  $\orb(x)=\orb(y)$  iff $G_x=G_y$.
\end{itemize}

\noindent
In case $\rank(G_x)=1$,
letting  $d$ denote the denominator of the
smallest nonzero positive rational element of $G_x$,

\begin{itemize}
\item[(ii)] If $d \leq 4,$  then
$c_x=1$.  For 
 all $y\in \mathbb R,$  $\orb(x)=\orb(y)$  iff $G_x=G_y$.

\medskip 
\item[(iii)] If $d > 4,$
 precisely  $\max(1,\phi(d)/2)$ distinct orbits
are classified by the same group
$G_x.$  These orbits are in one-one correspondence
with the totality of  pairs  $(G_x,c)$ with $c\in \mathbb Z,$
 $1 \leq  c < d/2$
and  $\gcd(c,d)=1.$  
 \end{itemize}
\end{corollary}


Concerning the topological properties of
$\orb(x),$  considerations from ergodic theory and
dynamics of flows in the homogeneous space
$SL(n,\mathbb R)/SL(n,\mathbb Z)$ lead to the well known
result  \cite[5.2]{dan} which, in the framework 
of the present paper, can be reformulated as follows:
{\it for each $n=2,3,\ldots$
and $x\in \Rn$,  $\orb(x)$ is dense iff
$\rank(G_x)\not=1$.}

 \section{Applications}

%

 \subsection*{Applications to lattice-ordered abelian groups with
 order unit}
In this section
we assume familiarity with lattice-ordered
abelian groups. We refer to \cite{bigkeiwol}  for background.
An {\it $\ell$-group}  (which is short for  ``lattice-ordered abelian group'')
is an abelian group $G$ equipped with a translation invariant lattice order.
If $G$ also has a distinguished (strong) {\it order unit},
i.e., an element   $u$ whose positive integer multiples eventually
dominate any element of $G$,   then  $(G,u)$ is known as
a {\it unital $\ell$-group}.  Morphisms in the category of unital
$\ell$-groups are
 {\it unital $\ell$-homomorphisms},
 i.e.,  group homomorphisms that also preserve the
lattice structure and the distinguished order unit.  
We write $(G,u)\cong(H,v)$ when $(G,u)$  and
$(H,v)$  are unitally $\ell$-isomorphic.

An {\it  ideal}  (``$\ell$-ideal'' in \cite{bigkeiwol})
 $\mathfrak i$  of $(G,u)$ is the kernel of a unital $\ell$-homomorphism
$\eta\colon (G,u)\to (H,v).$  In case  $\eta(G,u)$ is totally ordered, 
$\mathfrak i$ is said to be {\it   prime}  (``ir\-r\'e\-duc\-ti\-ble'' in \cite{bigkeiwol}).  
H\"older's  theorem,  \cite[2.6]{bigkeiwol}
 is to the effect that for every maximal ideal
$\mathfrak m$ of $(G,u)$   the quotient
$(G,u)/\mathfrak m$  is uniquely embeddable into
a unital $\ell$-subgroup of the additive group $\mathbb R$
equipped with the usual order and with the element 1 as the
order unit. The  set $\mu(G)$
 of maximal ideals of   $(G,u)$
is  a subset of the set of prime ideals. The hull-kernel topology
makes  $\mu(G,u)$ a   compact Hausdorff space,
\cite[13.2.6]{bigkeiwol}.  We then
say that  $\mu(G,u)$ is the {\it maximal spectral space} of
$(G,u)$.

Since the category of $\ell$-groups
is  equivalent to an equational class, \cite[3.9]{mun-jfa},
one can naturally introduce the free unital $\ell$-group $\mathcal F_n$ 
over $n$  free generators.  Every finitely generated
 unital $\ell$-group is
  (isomorphic to) the quotient $\mathcal F_n/\mathfrak i$
of some $\mathcal F_n$ by an ideal $\mathfrak i.$
The maximal spectral space of 
$\mathcal F_n$  is homeomorphic to the unit cube
$[0,1]^n$ with the product topology, so it makes sense to say that a maximal
ideal $\mathfrak m$ of $\mathcal F_n$ ``lies in the interior''
of $  \mu(\mathcal F_n)$.  See \cite[Proposition 3(i)]{mun-dcds} for details.  
As a particular case of  \cite[10.5.3]{bigkeiwol},  for every
maximal ideal $\mathfrak m$  of $(G,u)$ by the {\it germinal} ideal
$\mathfrak o(\mathfrak m)$,  we mean  the
intersection of all
prime ideals $\mathfrak p$  of $(G,u)$
contained in $\mathfrak m.$

\begin{theorem}
\label{theorem:ell}
Let $\mathcal F_2$  be the free unital
  $\ell$-group on two generators.
Let $\mathfrak m$  be a maximal ideal of 
$\mathcal F_2$ lying in the interior of
$\mu(\mathcal F_2)$.
Write  $G_\mathfrak m$ for  the underlying group of the
unital $\ell$-group  $\mathcal F_2/\mathfrak m.$

\begin{itemize}
\item[(i)]  $G_\mathfrak m$ is a subgroup
of the additive group
$\mathbb R$ 
of rank $\leq 3$ containing $1$.

\smallskip
\item[(ii)] 
If
$\rank(G_\mathfrak m)=3$  then
for  every maximal $\ell$-ideal $\mathfrak n$, 
$$
\mathcal F_2/\mathfrak m\cong \mathcal F_2/\mathfrak n
\Leftrightarrow
\mathcal F_2/\mathfrak o(\mathfrak m)\cong \mathcal F_2/\mathfrak o
(\mathfrak n).
$$ 
If
$\rank(G_\mathfrak m)$ equals $1$  then
for  every maximal $\ell$-ideal $\mathfrak n$
 lying in the interior of
$\mu(\mathcal F_2)$, 
$$
\mathcal F_2/\mathfrak m\cong \mathcal F_2/\mathfrak n
\Leftrightarrow
\mathcal F_2/\mathfrak o(\mathfrak m)\cong \mathcal F_2/\mathfrak o
(\mathfrak n).
$$ 

\smallskip
\item[(iii)]  If
$\rank(G_\mathfrak m)=2$,
let  $d$ be  the denominator of
the smallest positive
nonzero rational in  $G_\mathfrak m.$
Then the set    of isomorphism classes
of  unital $\ell$-groups of the form  $\mathcal F_2/\mathfrak o
(\mathfrak n),$ for $\mathfrak n$ a maximal ideal
 in the interior of $\mu(\mathcal F_2)$ such 
 that  $\mathcal F_2/\mathfrak n\cong \mathcal F_2/\mathfrak m,$
contains exactly  $\max\left(1,\frac{\phi(d)}{2}\right)$ elements.    
\end{itemize}
\end{theorem}

\begin{proof}
By  \cite[Proposition 2]{mun-dcds},
  $\mathcal F_2$ has a concrete representation as
   the unital $\ell$-group
of all piecewise linear continuous  functions  $f\colon 
[0,1]^2	\to \mathbb R$, where each piece of $f$ is a linear
(affine) polynomial  with integer coefficients.
The constant function 1 is the distinguished order unit of
$\mathcal F_2$.  The operations of $\mathcal F_2$  are
pointwise addition, subtraction, $\max,$ and $\min$.
Via the Yosida correspondences
$\mathfrak n\mapsto x_\mathfrak n$ and 
$x\mapsto \mathfrak m_x$
of  \cite[Proposition 3(i)]{mun-dcds}
  we  further identify 
the maximal spectral space $\mu(\mathcal F_2)$ with
$[0,1]^2$, in such a way that the maximal ideal  $\mathfrak m_x$
  is the set of all functions vanishing
at $x_\mathfrak m$, and the point $x_\mathfrak n$  
is the intersection of
the zerosets of all functions in $\mathfrak n.$
For every $\mathfrak p\in \mu(\mathcal F_2),$
the quotient map $f\in \mathcal F_2\mapsto f/\mathfrak p$
amounts to evaluation at $x_\mathfrak p,$
$$
f/\mathfrak p = f(x_\mathfrak p).
$$
Then the germinal ideal $\mathfrak o(\mathfrak n_x)$
is the set of all functions vanishing on an open neighbor
of $x_\mathfrak p.$

 \smallskip

  (i) By  H\"older's theorem,  
  $G_\mathfrak m$  is the subgroup of $\mathbb R$
   generated by the unit 1 together with the coordinates
   $x_1,x_2$  of
  $x$. So $G_\mathfrak m=G_x$ and  $\rank(G_\mathfrak m)
  \in \{1,2,3\}.$ The desired
  conclusion now
   follows from \cite[Proposition 3(ii)-(iii)]{mun-dcds}.

\smallskip
(ii) In case  $\rank(G_\mathfrak m)=3,$ 
the point  $x_\mathfrak m$ is totally irrational and thus
automatically  lies in the interior of $\mu(\mathcal F_2)$.
Each function  $f\in \mathcal F_2$  vanishing at
$x_\mathfrak m$ also vanishes on an open neighborhood
of  $x_\mathfrak m$. Therefore, the maximal quotient
$\mathcal F_2/\mathfrak m$ coincides with the
germinal quotient
$\mathcal F_2/\mathfrak o(\mathfrak m)$, and
we are done.
 
In case $\rank(G_\mathfrak m)=1$,
one first proves the following claim:
{\sl There is an isomorphism $\iota\colon \mathcal F_2/\mathfrak o(\mathfrak m)\cong
\mathcal F_2/\mathfrak o(\mathfrak n)$ iff  ($\mathfrak n$ lies in
the interior of $\mu(\mathcal F_2)$  and)
there exist rational triangles $T$ and $U$ with 
$ x_{\mathfrak m}\in {\rm int}(T)$  and 
$y_{\mathfrak n}\in {\rm int}(U)$ and a
map $\gamma\in\mathcal G_2$ satisfying
$ \gamma(x_{\mathfrak m})=x_{\mathfrak n}$ and $\gamma(T)=U$.}
By  {  Proposition \ref{proposition:rank1},}
a necessary and sufficient condition for such $\gamma$ to
exist is that  $G_\mathfrak m=G_\mathfrak n$. 
By H\"older's theorem, this
 is in turn equivalent to saying that the unital $\ell$-groups
$ \mathcal F_2/\mathfrak m$ and 
$ \mathcal F_2/\mathfrak n$ are unitally $\ell$-isomorphic.

\smallskip
  (iii) From
 { Lemma   \ref{lemma:cx}(iii)}
it follows that, whenever  $\rank(G_x)=2$,  the number of
orbits  classified by
the same group  $G_\mathfrak m$
is   $\max(1,\phi(d)/2)$, and { by Theorem
\ref{theorem:rank2}} these orbits
are in one-one correspondence with the possible values of
the parameter  $c_{x_\mathfrak m}$ of
{ Definition \ref{definition:cx}.}
\end{proof}

\medskip
 \subsection*{Applications to the Farey-Stern-Brocot AF C$^*$-algebra}
A  {\it (unital)  AF C$^*$-algebra}  $\mathfrak A$ is the norm closure
of an ascending sequence of finite-dimensional C$^*$-algebras,
all with the same unit.
In this section we require
  acquaintance with AF C$^*$-algebras
 and Elliott classification  
 \cite{eff,goo1},
 $$
 \mathfrak A\mapsto (K_{0}({\mathfrak A}), [1_{\mathfrak A}]),
 $$
 of  unital AF $C^{*}$-algebras in terms of
 countable unital dimension groups.
Here,  $K_0$  is a suitable
 order-theoretic expansion of  Grothendieck's
functor, and $[1_{\mathfrak A}]$ is the
$K_0$-image of the unit of $\mathfrak A.$

The Murray-von Neumann order of projections
in $\mathfrak A$ is a lattice iff $K_{0}({\mathfrak A})$
is an $\ell$-group.
{ Combining Corollary  \ref{corollary:onedim}
with Elliott's classification, 
in Theorem \ref{theorem:boca} } 
we will describe the mutual
 relationships between  the maximal
and the germinal ideals of 
the  Farey-Stern-Brocot
 $C^{*}$-algebra  $\mathfrak M_{1}$  introduced
  by the present author  in  
 \cite{mun-adv}  and recently rediscovered by Boca
in \cite{boc}  (also see \cite{eck} and \cite{mun-mjm}).
The Bratteli diagram of $\mathfrak M_{1}$
 is as follows:

	\begin{center}
	\unitlength=2.5mm
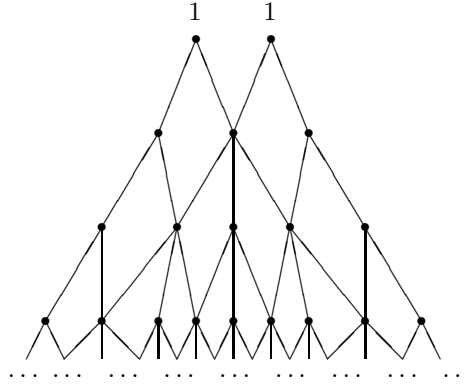
	\begin{picture}(46,25)(0,0)

		\put(17.6,21){1}
			\put(21.6,21){1}
		\put(18,20){\circle*{0.4}}
		\put(22,20){\circle*{0.4}}
		
			\put(16,15){\circle*{0.4}}
\put(16,15){\line(2, 5){2}}

			\put(20,15){\circle*{0.4}}
			\put(20,15){\line(-2, 5){2}}
				\put(20,15){\line(2, 5){2}}
				
			\put(24,15){\circle*{0.4}}
				\put(24,15){\line(-2, 5){2}}
			
			\put(13,10){\circle*{0.4}}
			\put(13,10){\line(3, 5){3}}
			
			\put(17,10){\circle*{0.4}}
\put(17,10){\line(-1, 5){1}}
\put(17,10){\line(3, 5){3}}

					\put(20,10){\circle*{0.4}}
						\put(20,10){\line(0, 5){5}}
						
						\put(23,10){\circle*{0.4}}
\put(23,10){\line(1, 5){1}}
\put(23,10){\line(-3, 5){3}}
						
							\put(27,10){\circle*{0.4}}
							\put(27,10){\line(-3, 5){3}}
							
\put(10,5){\circle*{0.4}}
	\put(10,5){\line(3, 5){3}}
	\put(10,5){\line(-1, -2){1}}
\put(10,5){\line(1, -2){1}}
	
\put(13,5){\circle*{0.4}}
\put(13,5){\line(0, 5){5}}
\put(13,5){\line(4, 5){4}}
	\put(13,5){\line(-1, -1){2}}
\put(13,5){\line(0, -2){2}}
\put(13,5){\line(1, -1){2}}

	\put(16,5){\circle*{0.4}}
	\put(16,5){\line(1, 5){1}}
	\put(16,5){\line(-1, -2){1.0}}
\put(16,5){\line(0, -2){2}}
\put(16,5){\line(1, -2){1}}
	
	\put(18,5){\circle*{0.4}}
		\put(18,5){\line(-1, 5){1}}
			\put(18,5){\line(2, 5){2}}
	\put(18,5){\line(-1, -2){1}}
\put(18,5){\line(0, -2){2}}
\put(18,5){\line(1, -2){1}}
	
\put(20,5){\circle*{0.4}}
\put(20,5){\line(0, 5){5}}
	\put(20,5){\line(-1, -2){1}}
\put(20,5){\line(0, -2){2}}
\put(20,5){\line(1, -2){1}}

\put(22,5){\circle*{0.4}}
		\put(22,5){\line(-2, 5){2}}
			\put(22,5){\line(1, 5){1}}
	\put(22,5){\line(-1, -2){1}}
\put(22,5){\line(0, -2){2}}
\put(22,5){\line(1, -2){1}}

	\put(24,5){\circle*{0.4}}
		\put(24,5){\line(-1, 5){1}}
	\put(24,5){\line(-1, -2){1}}
\put(24,5){\line(0, -2){2}}
\put(24,5){\line(1, -2){1}}
	
	\put(27,5){\circle*{0.4}}
	\put(27,5){\line(0, 5){5}}
\put(27,5){\line(-4, 5){4}}
	\put(27,5){\line(-1, -1){2}}
\put(27,5){\line(0, -2){2}}
\put(27,5){\line(1, -1){2}}
	
\put(30,5){\circle*{0.4}}
\put(30,5){\line(-3, 5){3}}
\put(30,5){\line(-1,-2){1}}
\put(30,5){\line(1,-2){1}}
\put(8,2){\ldots\,\,\ldots\,\,
\ldots\,\, \ldots\,\, \ldots\,\,
\ldots\,\, \ldots\,\, \ldots\,\, \ldots}

		\end{picture}
	\end{center}
	
\noindent The two top  (depth 0) vertices 
are labeled 1.
The label  
of any vertex
$v$ at depth $d=1,2,\ldots, $  is the sum of the labels
of the vertices at depth  $d-1$ connected to $v$ by an edge.

As proved in \cite{mun-adv}, an
 equivalent redefinition of $\mathfrak M_1$ is as follows:
$$
(K_{0}({\mathfrak M_1}), [1_{\mathfrak M_1}]) = \mathcal F_1.
$$ 
Now, up to unital $\ell$-isomorphism, 
$ \mathcal F_1$ is the unital $\ell$-group of
continuous piecewise linear functions
$f\colon [0,1]\to \mathbb R$, where each
piece of $f$  is a linear polynomial with integer coefficients,
\cite[Proposition 2]{mun-dcds}.
The constant function 1 is the distinguished order unit of
$\mathcal F_1$. 
As a particular case of a general well known fact
\cite[\S 3]{mun-mjm},  the transformation
\begin{equation}
\label{equation:quotient}
\mathfrak P\mapsto K_0(\mathfrak P) 
\end{equation}
  induces an order
preserving one-one map of  the ideals    of
$\mathfrak M_1$  onto
the  ideals   of $ \mathcal F_1$, in such a way
that
\begin{equation}
\label{equation:ebbene}
K_0\left(\frac{(\mathfrak M_1, [1_{\mathfrak M_1}])}{\mathfrak P}
\right)=\frac{K_0(\mathfrak M_1, [1_{\mathfrak M_1}])}{K_0(\mathfrak P)}.
\end{equation}
In particular,  $K_0$ sends primitive  (resp., maximal) ideals of
$\mathfrak M_1$  one-one onto
prime (resp., maximal) ideals of  $ \mathcal F_1$.
For every maximal ideal $\mathfrak I$ of  
 of $\mathfrak M_1$, the {\it germinal ideal}
$\mathfrak O(\mathfrak I)$  of $\mathfrak M_1$ at   $\mathfrak I$ is the
 intersection of all primitive ideals of 
 $\mathfrak M_1$ contained in $\mathfrak I.$
 The $K_0$-image of  $\mathfrak O(\mathfrak I)$ is the
 germinal ideal  $\mathfrak o(K_0(\mathfrak I)).$

For each irrational $\theta$, the   AF C$^*$-algebra
$\mathfrak F_\theta$  is defined by
\begin{equation}
\label{equation:irrational}
K_0(\mathfrak F_\theta, [1_{\mathfrak F_\theta}])=
(\mathbb Z+\mathbb Z\theta, 1)
\end{equation}
i.e., the subgroup of the additive group $\mathbb R$ generated
by 1 and $\theta$, equipped with the usual order, and with
the distinguished order unit 1. See \cite[p. 34, and p.65]{eff} 
 for further information.

As proved in \cite[3.4]{mun-mjm}, the maximal ideal 
space of  $\mathfrak M_1$ equipped with the hull-kernel topology
is canonically  homeomorphic to the unit real interval $[0,1].$

\begin{theorem}
\label{theorem:boca} 
Suppose $\mathfrak I$
is a maximal ideal of $\mathfrak M_1$.
\begin{itemize}
\item[(i)] If  
  $\mathfrak M_1/\mathfrak I\cong \mathfrak F_\theta$ 
for some irrational $\theta\in[0,1]$,  then  
$\mathfrak M_1/\mathfrak I\cong\mathfrak M_1/\mathfrak O(\mathfrak I).$

\smallskip
\item[(ii)] 
 If    
 $\mathfrak M_1/\mathfrak I\not\cong\mathfrak F_\theta$
  for every  irrational $\theta \in[0,1]$,   then
$\mathfrak M_1/\mathfrak I$ is isomorphic to
 the C$^*$-algebra of  linear operators
on $d$-dimensional Hilbert space, for some $d=1,2,\ldots$.
Let $\mathcal I$  be
 the set  of isomorphism classes
of  AF C$^*$-algebras of the form  
$\mathfrak M_1/\mathfrak O(\mathfrak L)$, for
 $\mathfrak L$  an arbitrary
 maximal ideal of  $\mathfrak M_1$
with
 $\mathfrak M_1/\mathfrak L \cong \mathfrak M_1/\mathfrak I$.
 Then $\mathcal I$
has exactly  $\max(1,{\phi(d)}/{2})$ elements.
\end{itemize}
\end{theorem}

\begin{proof} 
(i)   A routine verification shows that $K_0(\mathfrak I)$
is the set of all $f\in \mathcal F_1$ vanishing on $\theta.$ 
Since each linear piece of $f$ has integer coefficients
and $\theta$ is irrational, then $f$ automatically vanishes
on an open neighborhood of $\theta.$ So
$$
\mathcal F_1/\mathfrak m_\theta =
\mathcal F_1/\mathfrak o(\mathfrak m_\theta).
$$
Recalling (\ref{equation:quotient})-(\ref{equation:irrational}),  
we  conclude
$$
\mathfrak M_1/\mathfrak I\cong\mathfrak M_1/\mathfrak O(\mathfrak I).
$$

(ii)
Let $G$ be the underlying group
of   $K_0(\mathfrak M_1/\mathfrak I)$.
Let $d$ be the denominator of the smallest 
positive nonzero rational element of $G$.
{}From the complete description of
  $\mu(\mathfrak M_1)$ given 
in  \cite{boc} and \cite{mun-mjm}  it follows that 
$G$ has the form
$\mathbb Z\frac{1}{d}$, for some integer $d>0.$
An elementary calculation of $K_0$ 
shows that  $\mathfrak M_1/\mathfrak I$ is
 the C$^*$-algebra of  linear operators
on $d$-dimensional Hilbert space.
The rest now follows from
{  Corollary  \ref{corollary:onedim}}.
\end{proof}

\bigskip
\noindent{\bf Problem.} 
Prove or disprove:   For all $n=1,2,3,\dots$  and
$x,y\in \Rn,$
$\orb(x)=\orb(y)$
iff  $(G_x, c_x)=(G_y,c_y)$.

\bibliographystyle{plain}

\end{document}